\documentclass{amsart}[10pt]
\usepackage{etex}
 \usepackage[foot]{amsaddr}
\usepackage{graphicx}
\usepackage{epstopdf}

\usepackage{cite}
\usepackage{tikz, tikz-cd, tkz-graph}

\usepackage{xcolor}

\usepackage{hyperref}
\hypersetup{
  colorlinks,
  linkcolor={red!50!black},
  citecolor={blue!50!black},
  urlcolor={blue!80!black}
}

\usepackage[english]{babel}
\usepackage{amsmath}
\usepackage{amssymb}
\usepackage{amsthm}
\usepackage{xypic}
\usepackage{longtable}
\usepackage{array}

 \usepackage[enableskew]{youngtab}

\usepackage{epsfig}
\usepackage{hyperref}
\usepackage{enumitem}
\usepackage{booktabs}
 \usepackage{longtable} %used by appendix
 \usepackage{pdflscape} % used by appendix
 \usepackage{colortbl} % used by appendix
 \usepackage{arydshln} % used by appendix
 \usepackage{calc} % used by appendix
\usepackage[vcentermath]{genyoungtabtikz}
\Yboxdim{6pt}

\graphicspath{{./images/}}

\newtheorem{thm}{\bf Theorem}[section]
\newtheorem{eg}[thm]{\bf Example}
\newtheorem{prop}[thm]{\bf Proposition}
\newtheorem{cor}[thm]{\bf Corollary}
\newtheorem{rem}[thm]{\bf Remark}
\newtheorem{mydef}[thm]{\bf Definition}
\newtheorem{lem}[thm]{\bf Lemma}
\newtheorem{conjecture}[thm]{\bf Conjecture}

\newcommand{\CC}{\mathbb{C}}

\newcommand{\ZZ}{\mathbb{Z}}
\newcommand{\QQ}{\mathbb{Q}}
\newcommand{\NN}{\mathbb{N}}
\newcommand{\PP}{\mathbb{P}}

\newcommand{\br}{\mathbf{r}}

\DeclareMathOperator{\Rep}{\mathrm{Rep}}

\DeclareMathOperator{\Hom}{\mathrm{Hom}}
\DeclareMathOperator{\GL}{\mathrm{GL}}
\DeclareMathOperator{\Fl}{\mathrm{Fl}}

\DeclareMathOperator{\Gr}{\mathrm{Gr}}

\DeclareMathOperator{\rank}{\mathrm{rank}}

\DeclareMathOperator{\NE}{\mathrm{NE}}

\DeclareMathOperator{\Mat}{\mathrm{Mat}}
\DeclareMathOperator{\HR}{\mathrm{RH}}

\newcommand{\ab}{\text{\rm ab}}
\renewcommand{\emptyset}{\varnothing}

\newcolumntype{C}[1]{>{\centering\let\newline\\\arraybackslash\hspace{0pt}}m{#1}}
\title{A rim-hook rule for quiver flag varieties}
\author{Wei Gu}
\address[W.~Gu]{{Center of Mathematical Sciences and Applications \\ Harvard University\\ Cambridge, MA 02138}}
\email{weigu@cmsa.fas.harvard.edu}
\author{Elana Kalashnikov}
\address[E.~Kalashnikov]{{Department of Mathematics \\ 
Harvard University \\ 
Cambridge, MA 02138}}
\address[E.~Kalashnikov]{{Faculty of Mathematics \\ 
Higher School of Economics \\ 
Moscow, Russia}}
\email{kalashnikov@math.harvard.edu}

\begin{document}
\begin{abstract} The rim-hook rule for quantum cohomology of the Grassmannian allows one to reduce quantum calculations to classical calculations in the cohomology of the Grassmannian. We use the Abelian/non-Abelian correspondence for cohomology to prove a rim-hook removal rule for the cohomology of quiver flag varieties. Quiver flag varieties are generalisations of type A flag varieties; this result is new even in the flag case. This gives an effective way of computing products in their cohomology, reducing computations to that in the cohomology ring of the Grassmannian. We then prove a quantum rim-hook rule for Fano quiver flag varieties (including type A flag varieties). As a corollary, we see that the Gu--Sharpe mirror to a Fano quiver flag variety computes its quantum cohomology. 
\end{abstract}

\maketitle

\section*{Introduction}\label{sec:intro}
Consider a smooth projective variety that can be described as a GIT quotient $X_G=V//_{\theta}G$ of a $\CC$-vector space by a reductive algebraic group $G$, with stability condition $\theta \in \chi(G)$.  If $T \subset G$ is a maximal torus, then we can also view $\theta$ and $V$ as $T$-representations, and consider the GIT quotient $X_T:=V//_{\theta}T$. The projective variety $X_T$ is a toric variety, which we call the \emph{abelianisation} of $X_G$.  A theorem of Martin \cite{Martin2000} gives a description of the algebra structure of $H^*(X_G)$ via that of $H^*(X_T)$. Martin's result is striking when one considers the simplicity of the cohomology of toric varieties compared with the rich structure of $X_G$, examples of which include Grassmannians and type A flag varieties. This result, however, has so far had limited applications computationally. In this paper, we fill this gap: we use Martin's result to give a rim-hook removal rule for the cohomology  and quantum cohomology of $X_G$, when $X_G$ is a quiver flag variety. This is new even in the case of a type A flag variety.  

Rim-hook removals were introduced in \cite{bertram} as a powerful tool in computations in the quantum cohomology of the Grassmannian: essentially, they allow one to reduce quantum calculations to classical calculations, controlled by the Littlewood--Richardson coefficients. In this paper, we prove a rim-hook removal rule for the \emph{classical} cohomology ring of quiver flag varieties (including type A flag varieties). This gives a way to reduce computations in the cohomology of flag varieties and quiver flag varieties to computations in Grassmannians. Specifically, it allows one to compute the structure constants of the cohomology in a particular basis. We also prove a quantum rim-hook rule, which we discuss more below. 

Quiver flag varieties \cite{Craw2011, King1994} are a generalisation of type A flag varieties. Quiver flag varieties are constructed as smooth projective GIT quotients of the space of representations of acyclic quivers with a unique source, where the stability condition  $\theta$ is always the same. Let $Q=(Q_0,Q_1)$ be an acyclic quiver with a unique source, with arrows $Q_1$ and vertices $Q_0$. Label the vertices $Q_0:=\{0,\dots,\rho\}$, so that $n_{ij}$, the number of arrows from $i$ to $j$, vanishes when $i \geq j$. Fix a dimension vector $\br:=(r_0:=1,r_1,\dots,r_\rho)$.  The quiver flag variety associated to this quiver is denoted $M_\theta(Q,\br)$.   See Definition \ref{defn:quiverflagvariety} for the full definition. The abelianisation of a quiver flag variety $M_\theta(Q,\br)$ is another quiver flag variety, denoted $M_\theta(Q^{ab},\br)$. 

Craw \cite{Craw2011} also gives an equivalent construction of quiver flag varieties as towers of Grassmannian bundles. It is this second construction which is most relevant for this paper, so we sketch it now, in order to state the main results. We construct the quiver flag variety $M_\theta(Q,\br)$ inductively. Consider first the Grassmannian of quotients $M_1:=\Gr(\CC^{n_{01}},r_1)$, with its tautological quotient bundle $W_1$, and trivial line bundle $W_0$. Define $M_2$ to be the relative Grassmannian $\Gr(W_0^{\oplus n_{02}} \bigoplus W_1^{\oplus n_{12}},r_2)$. We label the relative quotient bundle on $M_2$ as $W_2$, and abuse notation to label the pullbacks of $W_0$ and $W_1$ as before.  Continue in this way, defining 
\[M_i:=\Gr(\bigoplus_{a \in Q_1, t(a)=i} W_{s(a)},r_i).\]
We let $s_i=\rank(\bigoplus_{a \in Q_1, t(a)=i} W_{s(a)}).$ Then Craw \cite{Craw2011} shows that $M_\theta(Q,\br)$ is isomorphic to $M_\rho$, a smooth projective variety of dimension $\sum_{i=1}^\rho r_i(s_i-r_i)$. 
For each $i=1,\dots,\rho$, and partition $\lambda$, let $s^i_\lambda$ denote the Schur polynomial in the Chern roots of $W_i$ indexed by $\lambda$. We label  the Chern roots $x_{i1},\dots,x_{ir_i}$. Consider the set $\mathcal{S}(Q)$, defined as
\[\mathcal{S}(Q):=\{s^1_{\lambda_1} \cdots s^\rho_{\lambda_\rho} \mid \lambda_i \text{ fits into a } r_i\times(s_i-r_i) \text{ rectangle}\}.\]
Our main result in classical cohomology is as follows. For more details, see Theorem \ref{thm:cohomology} and Proposition \ref{prop:hookrim}.
\begin{thm} The set of monomials $\mathcal{S}(Q)$ form a vector space basis for the cohomology ring $H^*(M_\theta(Q,\br))$. If $\mu=(\mu_1,\dots,\mu_{r_i})$ is a partition too wide to fit into an $r_i \times (s_i-r_i)$ box, then $s^i_\mu$ can be expanded as a signed sum of rim-hook removals of $\mu$ with coefficients which are polynomials in the Chern classes of the $W_j$. 
\end{thm} 
In order to obtain the quantum rim-hook removal rule, we show the following description of the quantum cohomology of a Fano quiver flag variety (for full details, see Theorem \ref{thm:arrowquantum} below). Let $M_\theta(Q,\br)$ be a Fano quiver flag variety. Let $I^{ab}_q$ be the ideal generated by 
\begin{equation} \label{eq:qr} \prod_{\substack{a \in Q_1\\ t(a)=i}}\prod_{k=1}^{r_{s(a)}} (-x_{s(a)k}+x_{t(a)j})-(-1)^{r_i-1}q_i \prod_{\substack{a \in Q_1\\s(a)=i}}\prod_{k=1}^{r_{t(a)}} (-x_{s(a)j}+x_{t(a)k}).\end{equation}
\begin{thm}\label{thm:intro2} Let $M_\theta(Q,\br)$ be a Fano quiver flag variety. Then
\[QH(M_\theta(Q,\br)) \cong \CC[x_{ij},q_i: 1 \leq i \leq \rho, 1 \leq j \leq r_i]^W/I_q\]
where $I_q$ is the ideal generated by $f\in \CC[x_{ij},q_i: 1 \leq i \leq \rho, 1 \leq j \leq r_i]^W$ such that $\omega f \in I^{ab}_q$ and $\omega=\prod_{i=1}^\rho \prod_{j<k}(-x_{ij}+x_{ik})$. The group $W=\prod S_{r_i}$ acts by each factor permuting the $x_{i1},\dots,x_{ir_i}$. 
\end{thm}
This is exactly the quantum generalisation of the theorem in \cite{Martin2000}, appearing below as Theorem \ref{thm:martin}. The quantum relations \eqref{eq:qr} allows one to see the quantum cohomology of a quiver flag variety as that of a product of Grassmannians (one for each vertex) together with two different types of corrections: at the $i^{th}$ vertex, one set of corrections is classical and arises from arrows \emph{into} the vertex, and the other set is quantum, arising from arrows \emph{out} of the vertex. 

As a corollary, we see that the critical locus of the Gu--Sharpe \cite{GuSharpe} mirror exactly computes the quantum cohomology of a Fano quiver flag variety (see \S \ref{sec:gusharpe}).  

Various quantum versions of the Abelian/non-Abelian correspondence have appeared in the literature, both conjectural and proven: \cite{CiocanFontanineKimSabbah2008}, \cite{Webb2018}, \cite{kalashnikov},\cite{GuSharpe}. The strongest statement is the Abelian/non-Abelian correspondence conjecture  \cite{CiocanFontanineKimSabbah2008} (appearing as Conjecture \ref{conj:original} below); this was proven by Webb \cite{Webb2018} for GIT quotients of vector spaces of Fano index at least 2 with certain nice torus actions. We show that for Fano quiver flag varieties, the condition  of \cite{Webb2018} on the Fano index can be dropped; that is, we show that the mirror map for a Fano quiver flag variety is trivial (just as it is for toric varieties). It then follows from \cite{Webb2018} that the conjecture of \cite{CiocanFontanineKimSabbah2008} holds for Fano quiver flag varieties. This is the main ingredient in the proof of Theorem \ref{thm:intro2}. 

 Theorem \ref{thm:intro2} also implies the quantum version of the rim-hook removal rule for a Fano quiver flag variety. 
\begin{thm}[see Theorem \ref{thm:qhookrim} for details] The set $\mathcal{S}(Q)$ defined above form a vector space basis for the small quantum cohomology ring $QH^*(M_\theta(Q,\br))$. If  $\mu=(\mu_1,\dots,\mu_{r_i})$ is a partition too wide to fit into an $r_i \times (s_i-r_i)$ box, then $s^i_\mu$ can be expanded as a signed sum of rim-hook removals of $\mu$ with coefficients which are polynomials in the Chern classes of the $W_j$ and the quantum parameters. 
\end{thm} 

\paragraph{\emph{Plan of the paper}}
We begin by stating Martin's theorem precisely, and, as a warm-up, explaining the relationship between his description of the cohomology of the Grassmannian and other descriptions. We then review quiver flag varieties, and prove the first theorem above. We then give a brief introduction to quantum cohomology, before deriving the second theorem from \cite[Conjecture 3.7.1]{CiocanFontanineKimSabbah2008} for Fano quiver flag varieties.  In the last part of the paper, we show that \cite[Conjecture 3.7.1]{CiocanFontanineKimSabbah2008} holds for Fano quiver flag varieties, by showing that it is implied by the Abelian/non-Abelian correspondence for $I$-functions of \cite{Webb2018}. 

\subsection*{Acknowledgments}
The second author is very grateful for helpful conversations with Anders Buch, Tom Coates, Alastair Craw, Allen Knutson, and Lauren Williams.

\section{Martin's theorem and the Grassmannian}
In this section, we state Martin's theorem and present an explicit isomorphism between the two descriptions of the cohomology of the Grassmannian. We describe how to use Martin's theorem to carry out multiplication in the cohomology ring of  quiver flag variety.
\subsection{The statement}
We state Martin's theorem only in the generality we need.  Let $V$ be a complex vector space, and $G$ a reductive algebraic group acting on $V$. Let $\theta \in \chi(G)$ be a character of $G$. As $V$ is a vector space, a character is equivalent to a $G$-linearised line bundle (that is, a line bundle with a compatible $G$ action) on $V$, so this data defines a GIT quotient $X_G:=V//_\theta G=V^{ss}(\theta)/G$. Here $V^{ss}(\theta)$ represents the semi-stable elements of $\theta$.  From now on, we assume that $X_G$ is smooth: that is, we assume that $V^{ss}(\theta)=V^{s}(\theta)$ and $G$ acts freely on $V^{ss}(\theta)$. We can obtain a class of vector bundles, called \emph{representation theoretic vector bundles} on $X_G$ via this description. If $E$ is another representation of $G$, then consider the quotient $E_G:=V^{ss}(G)\times E/G$, where $G$ acts on $V^{ss}(G)\times E$ via the diagonal action. The quotient $E_G$ is a vector bundle on $X_G$ via the natural projection map.

Let $T$ be a maximal torus of $G$. The character $\theta$ is naturally a character of $T$ as well, so we can define a GIT quotient $X_T:=V//_\theta T$. We assume that there are no strictly semi-stable elements of this GIT quotient. The two GIT quotients $X_G$ and $X_T$ are related by an intermediate space  $V^{ss}(G)/T$ and two maps. Let $i:V^{ss}(G)/T \to V^{ss}(T)/T=X_T$ be the natural inclusion, and $\pi:V^{ss}(G)/T \to V^{ss}(G)/G=X_G$ be the projection. 

Denote the set of roots of $G$ to be $R$; that is, $R$ is the set of weights of $T$ that appear in the representation of $T$ on $\mathfrak{g}\otimes \CC$. Let $R^+$ be a choice of positive roots, and let $R^-$ denote the corresponding negative roots. A root $\alpha \in R$ defines an action of $T$ on $\CC$, and hence a line bundle on $X_T$. We denote this line bundle $L_\alpha$. Let $W$ denote the Weyl group of $T$. Define $\omega=\sqrt{\frac{(-1)^{|R^+|}}{|W|}}\prod_{\alpha \in R^+} c_1(L_\alpha)$. 
\begin{thm}\label{thm:martin}\cite{Martin2000,stromme} There is a surjective map of algebras $H^*(X_T,\QQ)^W \to H^*(X_G,\QQ)$ whose kernel is given by the ideal $(a: a \in H^*(X_T,\QQ)^W \mid a \cdot \omega =0)$.
\end{thm}
This completely describes $H^*(X_G,\QQ)$ as an algebra.

Martin also gives a theorem for computing integrals on $X_G$. Let $\alpha \in H^*(X_G)$ be a cohomology class, and define a \emph{lift} of $\alpha$, denoted $\tilde{\alpha}$, as any class such that $i^*(\tilde{\alpha})=\pi^*(\alpha)$. Lifts are not unique, unless $\alpha \in H^2(X_G)$ and the unstable locus has codimension at least two. 
\begin{thm}\cite{Martin2000}\label{thm:int} Let $\alpha\in H^*(X_G)$. Then 
\[\int_{X_G} \alpha = \int_{X_T} \tilde{\alpha} \cup \omega^2.\]
\end{thm}
\subsection{Martin's description of the cohomology of the Grassmannian}
We give this description for the Grassmannian in detail, presenting a simple isomorphism between this description and the usual one coming from Schubert calculus.

The Grassmannian $\Gr(n,r)$ of $r$-dimension quotients \footnote{As we consider the Grassmannian of quotients, our conventions are transpose to some of the literature.} of $\CC^n$ can be constructed as the GIT quotient of $V:=\Mat(r\times n)$ by $G:=\GL(r)$. The group $G$ acts by matrix multiplication on the right, and the stability condition is $\theta(g)=\det(g), g \in \GL(r)$. Let $T$ be the diagonal torus of $G$. The Weyl group is the symmetric group on $r$ elements.  It is easy to see that the abelianisation is $X_T:=(\PP^{n-1})^{\times r}$. The cohomology ring of $X_T$ is 
\[\CC[x_1,\dots,x_r]/(x_1^n,\dots,x_r^n),\]
where $x_i$ is the hyperplane class of the $i^{th}$ factor of $\PP^{n-1}$. The Weyl group acts by permuting the $x_i$. In this notation, $\omega:=\prod_{i<j} x_i-x_j.$ Then Martin's theorem states that $H^*(\Gr(n,r),\CC)$ is isomorphic to the quotient of the symmetric ring of polynomials in $x_1,\dots, x_r$ by the ideal $I$ where $a \in I$ if
\begin{equation}\label{eqn:Iideal} a \omega \in (x_1^n,\dots,x_r^n).\end{equation}

There are several other descriptions of the cohomology of the Grassmannian, we give one of these, as it is used to describe Schubert calculus and quantum Schubert calculus. Let $\lambda=(\lambda_1,\dots,\lambda_k)$ be a partition (that is, a decreasing sequence of non-negative integers). This describes a Young tableau, where the $i^{th}$ row has $\lambda_i$ boxes. For example, the partition $(2,1,1)$ corresponds to 
\[\yng(2,1,1).\]
We denote the tableau $\lambda$ as well. The Schur polynomial $s_\lambda$ in $r$ variables associated to $\lambda$ can be described in either of two ways. First, a semi-standard Young tableau (SSYT) is a filling of each box in the tableau with a number from $1,\dots, r$ such that rows are increasing and columns are strictly increasing. For each SSYT  $A$ of $\lambda$, let $m_A$ be the monomial in $x_1,\dots,x_k$ whose degree in $x_i$ is the number of times $i$ appears in $A$. Then $s_\lambda:=\sum_{A} m_A.$ If the length of $\lambda$ is more than $r$, $s_\lambda$ is 0.

 One can also describe $s_\lambda$ as follows. We can assume it has length exactly $r$ by adding zeroes.  Then define $s_\lambda$ as
\[s_\lambda \prod_{i<j}(x_i-x_j)= \det
\begin{bmatrix}
x_1^{r-1+\lambda_1}& \cdots & x_r^{r-1+\lambda_1} \\
\vdots & & \vdots \\
x_1^{1+\lambda_r}& \cdots & x_r^{1+\lambda_r} \\
x_1^{\lambda_r}& \cdots & x_r^{\lambda_r} \\
\end{bmatrix} \]
These two descriptions coincide.  The polynomial $s_\lambda$ is symmetric, and the set of all Schur polynomials generate the symmetric algebra on $r$ variables. The cohomology ring of the Grassmannian can be identified with the quotient of $\QQ[x_1,\dots,x_n]^W$ by the ideal $J$ generated by all Schur polynomials where the corresponding Young tableau has either width greater than $n-r$ or height greater than $r$ (that is, it does not fit into a $r \times n-r$ box). 

We now give a simple description of an isomorphism between these two descriptions of the cohomology of the Grassmannian (of course, that they are isomorphic is the content of Martin's theorem): this was pointed out to us by Lauren Williams.
\begin{lem}\label{lem:grassmannian} Let $I$ be the ideal from Equation \eqref{eqn:Iideal} and $J$ be as above. Then $I=J$.
\end{lem}
\begin{proof}
We first show that $J \subset I$. It is well-known that $J$ is in fact generated by $s_\lambda$ where $\lambda=(k)$, $k=n-r+1,\dots,n$, so it suffices to show that $s_{(k)} \in I$. By the second description of Schur polynomials, 
\[\omega s_{(k)}= \det \begin{bmatrix}
x_1^{r-1+k}& \cdots & x_r^{r-1+k} \\
\vdots & & \vdots \\
x_1& \cdots & x_r\\
1 & \cdots & 1 \\
\end{bmatrix} .\]
As  $k>n-r+1$, each monomial in the first row has degree at least $n$. Each monomial in the expansion of the determinant is thus a multiple of $x_i^n$ for some $i$, and so $s_{(k)} \in I$.

In the reverse direction, let $f$ be a symmetric polynomial in $I$. It can be written as a sum of Schur polynomials, and by assumption $\omega f \in (x_i^n \mid i=1,\dots,r)$. Note that we can write $f=g+ f'$ where $g$ is some element of $J$ and $f'=\sum_{\lambda} s_{\lambda}$, where all $\lambda$ appearing in the sum fit into a $r \times n-r$ box.  Note that $f' \in I$. Suppose $f' \neq 0$. Then writing $\omega f'=\sum_{\lambda}\omega s_{\lambda}$ as a sum of determinants, each $\lambda$ contributes a term of the form 
\[\begin{bmatrix}
x_1^{r-1+\lambda_1}& \cdots & x_r^{r-1+\lambda_1} \\
\vdots & & \vdots \\
x_1^{1+\lambda_{r-1}}& \cdots & x_r^{1+\lambda_{r-1}} \\
x_1^{\lambda_r}& \cdots & x_r^{\lambda_r} \\
\end{bmatrix}. \]
Note that the degree in $x_i$ of any monomial in the determinant is at most $r-1+\lambda_1<n$, by assumptions on $\lambda$. As $\omega f' \in (x_i^n \mid i=1,\dots, r)$, this implies $f'=0.$
\end{proof}
\section{Cohomology of quiver flag varieties}
We briefly recall the construction of quiver flag varieties.  Quiver flag varieties are generalizations of Grassmannians and type A flag varieties, introduced by Craw \cite{Craw2011}.  Like flag varieties, they are GIT quotients satisfying the assumptions above: this implies (via \cite{King1994}) that they are fine moduli spaces.  A quiver flag variety $M_\theta(Q,\br)$ is determined by a quiver $Q$ and a dimension vector $\br$.  The quiver $Q$ is always assumed to be finite and acyclic, with a unique source.  Let $Q_0 = \{0,1,\ldots,\rho\}$ denote the set of vertices of $Q$ and let $Q_1$ denote the set of arrows.  Without loss of generality, after reordering the vertices if necessary, we may assume that $0 \in Q_0$ is the unique source and that the number $n_{ij}$ of arrows from vertex $i$ to vertex $j$ is zero unless $i<j$.  Write $s$,~$t :  Q_1 \to Q_0$ for the source and target maps, so that an arrow $a \in Q_1$ goes from $s(a)$ to $t(a)$.  The dimension vector $\br = (r_0,\ldots,r_\rho)$ lies in $\NN^{\rho+1}$, and we insist that $r_0 = 1$. $M_\theta(Q,\br)$ is defined to be the moduli space of $\theta$-stable representations of  the quiver $Q$ with dimension vector $\br$. Here $\theta$ is a fixed stability condition defined below. 

Consider the vector space
\[
\Rep(Q,\br) =\bigoplus_{a \in Q_1}\Hom(\CC^{r_{s(a)}},\CC^{r_{t(a)}})
\]
and the action of $\GL(\mathbf{r}):=\prod_{i=0}^\rho \GL(r_i)$ on $\Rep(Q,\br)$ by change of basis.  The diagonal copy of $\GL(1)$ in $\GL(\br)$ acts trivially, but the quotient $G := \GL(\br)/\GL(1)$ acts effectively; since $r_0 = 1$, we may identify $G$ with $\prod_{i=1}^\rho \GL(r_i)$.  
\begin{mydef} \label{defn:quiverflagvariety} The quiver flag variety $M_\theta(Q,\br)$ is the GIT quotient $\Rep(Q,\br)/\!\!/_\theta\, G$, where the stability condition $\theta$ is the character of $G$ given by 
\begin{align*}
  \theta(g) = \prod_{i=1}^\rho \det(g_i), && g = (g_1,\ldots,g_\rho) \in \prod_{i=1}^\rho \GL(r_i).
\end{align*}
\end{mydef}
 For the stability condition $\theta$, all semistable points are stable. 
 
Quiver flag varieties are fine moduli spaces and hence carry universal bundles: let $W_i$ denote the $i^{th}$ universal bundle; it is a vector bundle of rank $r_i$. This is the same vector bundle as described in the introduction, via constructing $M_\theta(Q,\br)$ as a tower of Grassmannian bundles. 
\begin{eg}Consider the quiver flag variety associated to the quiver
\begin{center}
  \includegraphics[scale=0.5]{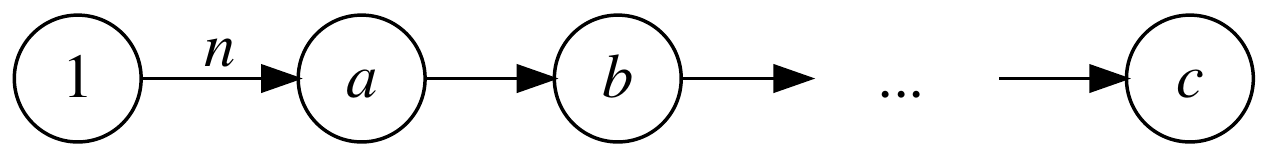}.
\end{center}
Here the label inside a vertex is its dimension, and labels above arrows indicate multiplicity (which is assumed to be 1 if not labelled). The associated quiver flag variety is the flag variety of quotients $Fl(n;a,b,\dots,c)$ in $\CC^n$. 
\end{eg}
 
 Different quivers can represent isomorphic quiver flag varieties. For example, via the construction as a tower Grassmann bundles if $s_i=r_i$ (recall that $s_i=\rank(\oplus_{a\in Q_1, t(a)=i} W_{s(a)})$) then $M_i=M_{i-1}$, and so we can represent $M_\theta(Q,\br)$ as a quiver flag variety with one less vertex. Therefore, we always assume that $s_i-r_i>0$. The paper \cite{kalashnikov} explores other equivalences: one equivalence is given by \emph{grafting}, where one takes any subquiver connected to the whole quiver only by some dimension 1 vertex, and moves the sub-quiver to the source. A quiver which has no graftable sub-quivers is called \emph{graph-reduced}. In \cite[Corollary 2.4.16]{thesis}, it's shown that any graph-reduced quiver flag variety has ample cone the positive orthant. Therefore, any quiver flag variety is equivalent to a quiver flag variety with ample cone the positive orthant, and for simplicity, we assume that all quiver flag varieties are graph reduced. We can now assume that a quiver flag variety is ample if and only if $s_i-s_i'>0$ for all $i$, where 
 \begin{equation} \label{eqn:si} s_i:=\sum_{a\in Q_1, t(a)=i} r_{s(a)}, \quad s_i'=\sum_{a\in Q_1, t(a)=i} r_{s(a)}.\end{equation}

If the dimension vector $\br=(1,\dots,1)$, then $M_\theta(Q,\br)$ is a toric variety; we call such quiver flag varieties \emph{toric quiver flag varieties}. These are smooth projective toric varieties: they are GIT quotients of $\CC^{Q_1}$ by $K\cong (\CC^*)^{\rho}$, where $K$ acts by weight $D_a \in \chi(K)$, where $D_a=-e_{s(a)}+e_{t(a)}$, where $e_i$ is the standard basis element. We set $e_0=0$. The weights give the projection in the following exact sequence: the first map is just the kernel:
\[0 \to \ZZ^{|Q_1|-\rho} \to \ZZ^{|Q_1|} \to \ZZ^{\rho} \to 0.\]
The cohomology of a toric variety can be described via this exact sequence (see, for example, \cite{audin}). In the case of a toric quiver flag variety, we re-write this in terms of the quiver.
\begin{prop} Let $M_\theta(Q,r)$ be a toric quiver flag variety. Then 
\[H^*(M_\theta(Q,\br),\QQ) \cong \QQ[x_1,\dots,x_\rho]/I,\]
where $I$ is generated by the following elements, for each $i \in \{1,\dots,\rho\}$:
\[\prod_{a \in Q_1, t(a)=i} (-x_{s(a)}+x_{t(a)}). \label{eqn:abrelation}\]
Here we set $x_0=0$. 
\end{prop}
\begin{proof}
This is a straightforward application of the usual description of the cohomology of a toric variety, translated into GIT data (see for example \cite{crepantconjecture}), together with the observation that the maximal subsets $S \subset Q_1$ such that $\theta$ is \emph{not} contained in the cone over the $D_a, a \in S$, are precisely the collections, for each $i=1,\dots\rho$
\[\{a \in Q_1 \mid t(a) \neq i\}.\]
\end{proof}
We can now use this description to describe the cohomology of a general quiver flag variety. It is shown in \cite{kalashnikov} that the abelianisation of a quiver flag variety is in fact a toric quiver flag variety: it is the toric quiver flag variety associated to the \emph{abelianised quiver}. The abelianised quiver $Q^{ab}$ has  $r_i$ vertices for every vertex $i$ in $Q$; the number of arrows between a vertex coming from $i \in Q_0$ and a vertex from $j \in Q_0$ is $n_{ij}$. For example, the quiver
\begin{center}
\includegraphics[scale=0.5]{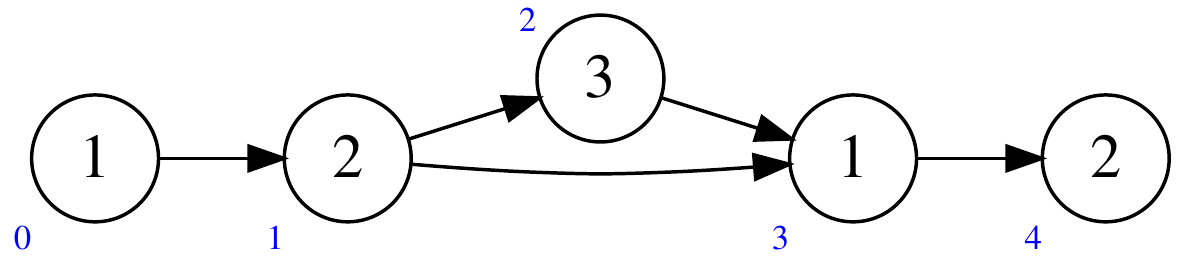}
\end{center}
has abelianisation
\begin{center}
\includegraphics[scale=0.5]{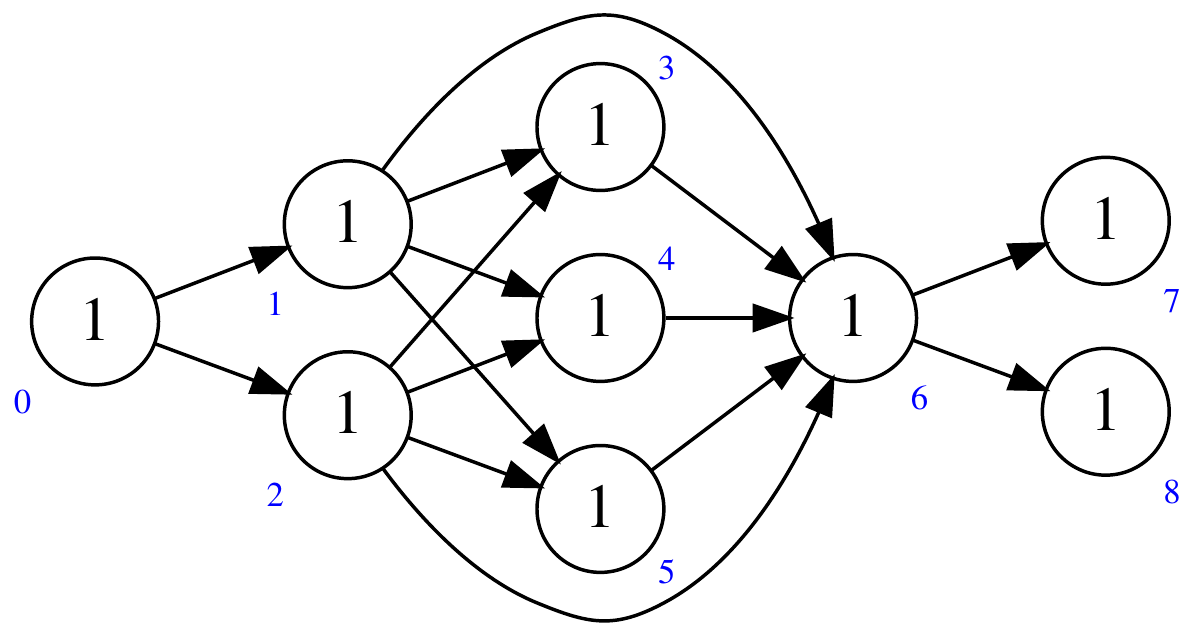}.
\end{center}
Then Martin's theorem immediately implies that
\begin{prop} Let $M_\theta(Q,r)$ be a quiver flag variety. Let $W=\prod_{i=1}^\rho S_{r_i}$ be the product of symmetric groups. Then 
\[H^*(M_\theta(Q,\br),\QQ) \cong \QQ[x_{11},\dots,x_{1 r_i},\dots, x_{\rho 1},\dots,x_{\rho r_\rho}]^W/(f : \omega f \in I^{ab}),\]
where $I^{ab}$ is generated by the following elements, for each $i=1,\dots,\rho$, and $j=1,\dots,r_i$:
\begin{equation}\prod_{a \in Q_1, t(a)=i}\prod_{k=1}^{r_{s(a)}} (-x_{s(a)k}+x_{t(a)j}).\label{eqn:Iab}\end{equation}
Here we set $x_{01}=0$, and $\omega=\prod_{i=1}^\rho \prod_{1\leq j<k \leq r_i} (x_{ij}-x_{ik})$.
\end{prop}
The $x_{ij}$ are Chern roots of the $i^{th}$ tautological bundle on $M_\theta(Q,\br)$. 

As stated, it is hard to work computationally with this proposition. However, using the same idea as in Lemma \ref{lem:grassmannian}, we can describe a vector space basis and give an effective algorithm for computing products of elements of the basis. 
\begin{rem}In the case of a non-Grassmannian flag variety, this basis is far from being the Schubert basis. However, it is the natural basis to consider from the perspective of the Abelian/non-Abelian correspondence.
\end{rem}
We first describe the generators. Given a quiver flag variety $M_\theta(Q,\br)$, let 
\[s_i:=\sum_{a \in Q_1, t(a)=i}r_{s(a)}.\]
Notice that the degree of the relation \eqref{eqn:abrelation} is of degree $s_i$.  

We first describe the candidate generators. For a partition $\lambda$,  the image of the Schur polynomial associated to $\lambda$ in the variables $x_{i1},\dots, x_{i r_i}$ in $H^*(M\theta(Q,\br))$ is denoted $s^i_\lambda$ .  For example, if $\lambda$ is the length $k$ partition $(1,\dots,1)$, $s^i_\lambda=c_k(W_i)$, and $s^i_{\emptyset}=1$. Set $s^0_\lambda=0$ for any partition $\lambda \neq \emptyset$. Motivated by the Grassmann bundle structure of $M_\theta(Q,\br)$, consider all $s^i_\lambda$ such that $\lambda$ fits into an $r_i \times (s_i-r_i)$ box. 
\begin{thm} \label{thm:cohomology} The $s^i_\lambda$ generate $H^*(M_\theta(Q,\br))$ as an algebra.
\end{thm}
We delay the proof of this theorem momentarily, and note that it implies
\begin{cor} The cohomology of $M_\theta(Q,\br)$ has a vector space basis given by products
\[s^1_{\lambda_1} \cdots s^\rho_{\lambda_\rho}\]
where $\lambda_i$ varies through all partitions fitting inside an $r_i \times (s_i-r_i)$ box. 
\end{cor}
\begin{proof} Following from the Grassmann bundle structure, the cohomology of $M_\theta(Q,\br)$ has dimension $\prod_{i=1}^\rho {s_i \choose r_i}$. The theorem implies that all such products generate the cohomology; as there are precisely  $\prod_{i=1}^\rho {s_i \choose r_i}$ such products, they must be a basis.
\end{proof}
We denote the set of elements of this basis $\mathcal{S}(Q)$. As we will see, the structure constants are not positive.

To prove the theorem, we will use the rim-hook removal operation on a partition that first arose in the quantum cohomology of the Grassmannian \cite{bertram} \footnote{There are slightly more general rim-hooks, but these are the only ones we will need.}. A rim-hook of a partition $\lambda$ of length $n$ is a connected path of $n$ boxes in $\lambda$ starting from the top right box, and staying within the rim or boundary of $\lambda$. For example, the partition
$$\yng(4,2,2)$$
has a rim-hook of length  $5$ given by the marked boxes:
\[\young(~\bullet \bullet \bullet,~\bullet,~\bullet).\]
The following is not a rim-hook because it is not continuous:
\[\young(~~\bullet \bullet,~\bullet,~\bullet).\]
The following is not a rim-hook because it doesn't stay on the rim:
\[\young(\bullet \bullet \bullet \bullet,~~,~~).\]
The height of a rim-hook is the number of rows in which it appears: in the example above, it is of height 3. Notice that rim-hook of a given length is unique. A rim-hook removal is the removal of a rim-hook from the partition: the resulting diagram may or may not be a partition. A rim-hook removal is said to \emph{exist} if the removal is still a partition. Call this partition $\lambda'$. In the above example, the rim-hook removal of length 4 does not exist, but the one of length 5 does. Then if $s_\lambda$ is a Schur polynomial, we set
\[\HR^n(s_\lambda)=(-1)^{h+1}s_{\lambda'}\]
if the length $n$ rim-hook removal exists, and set $\HR^n(\lambda)=0$ otherwise. Here $h$ is the height of the length $n$ rim-hook removal.
The authors \cite{bertram} show that  if $\lambda_1 \geq n$, then $s_\lambda$ and $\HR^n(s_\lambda)$ can also be related via the determinant construction of $s_\lambda$ as follows:
\[\prod_{i<j}(x_i-x_j) \HR^{n}(s_{\lambda})=\det
\begin{bmatrix}
x_1^{r-1+\lambda_1-n}& \cdots & x_r^{r-1+\lambda_1-n} \\
\vdots & & \vdots \\
x_1^{1+\lambda_r}& \cdots & x_r^{1+\lambda_r} \\
x_1^{\lambda_r}& \cdots & x_r^{\lambda_r} \\
\end{bmatrix}.\]
Now consider a quiver flag variety $M_\theta(Q,\br)$. Re-write the equation \eqref{eqn:Iab} as 
\[\prod_{a\in Q_1,t(a)=i} (\sum_{k=0}^{r_{s(a)}} (-1)^{k} x_{ij}^{r_{s(a)}-k} s^{s(a)}_{(k)^t})=0.\]
Here $s^i_{(k)^t}$ is the $k^{th}$ elementary symmetric polynomial, the Schur polynomial associated to the transpose of the partition $(k)$. Further expand this as 
\begin{equation} \label{eqn:lowerdegree}\sum_{k=0}^{s_i} \sum_{\sum k_a=k} (-1)^{k} x_{ij}^{s_i-k} \prod_{a \in Q_1, t(a)=i} s^{s(a)}_{(k_a)^t}=0.\end{equation}
The sum over $\sum k_a=k$ is understood to be the sum over all choices $0 \leq k_a \leq r_{s(a)}$ such that $\sum k_a=k.$ 

In the proposition below, by the notation $s^i_{\lambda+\{s_i-k\}}$ we mean the partition obtained from $\lambda$ by adding $s_i-k$ to its first entry. Recall that we have set $s^0_\lambda=0$ for any partition $\lambda\neq \emptyset$.
\begin{prop}[Rim-hook removal rule]\label{prop:hookrim}
Let $\lambda$ a partition such that $\lambda_1 \geq s_i-r_i+1$. Then 
\[s^i_\lambda=\sum_{k=1}^{s_i} \sum_{\sum k_a=k} (-1)^{k+1} \HR^{s_i}(s_{\lambda+\{s_i-k\}}) \prod_{a\in Q_1, t(a)=i} s^{s(a)}_{(k_a)^t}.\]
\end{prop}
\begin{proof}
Note that 
\[\prod_{1\leq j<k \leq r_i} (x_{ij}-x_{ik}) s^i_\lambda=\det
\begin{bmatrix}
x_{i1}^{r_i-1+\lambda_1}& \cdots & x_{ir_i}^{r_i-1+\lambda_1} \\
\vdots & & \vdots \\
x_{i1}^{1+\lambda_{r_i-1}}& \cdots & x_{ir_i}^{1+\lambda_{r_i-1}} \\
x_{i1}^{\lambda_{r_i}}& \cdots & x_{i r_i}^{\lambda_{r_i}} \\
\end{bmatrix}.\]
Replace the $x_{ij}^{r_i-1+\lambda_1}=x_{ij}^{s_i}x_{ij}^{r_i-1+\lambda_1-s_i}$ by replacing $x^{s_i}$ using \eqref{eqn:lowerdegree}. Using multi-linearity of the determinant in rows and the formula for rim-hook removal, we obtain that the above expression is equal to 
\[\prod_{1\leq j<k \leq r_i} (x_{ij}-x_{ik})  \sum_{k=1}^{s_i} \sum_{\sum k_a=k} (-1)^{k+1} \HR^{s_i}(s_{\lambda+\{s_i-k\}}) \prod_{a\in Q_1, t(a)=i} s^{s(a)}_{(k_a)^t}.\]
\end{proof}
\begin{eg}Suppose $M_\theta(Q,\br)$ is the flag variety of quotients of $\CC^n$, $\Fl(n;r_1,\dots,r_\rho)$.  
Then $s^i_\lambda$ is the Schur polynomial in the Chern roots of the $i^{th}$ tautological quotient bundle on the flag variety (the one of rank $r_i$). The rim-hook removal rule in this case states that if $\lambda_1 \geq r_{i-1}-r_i+1$, then 
\[s^i_\lambda=\sum_{k=1}^{r_{i-1}} (-1)^{k+1} \HR^{r_{i-1}}(s_{\lambda+\{r_{i-1}-k\}})s^{i-1}_{(k)^t}.\]
By convention, $r_{0}=n$. Note that this implies that $s^1_\lambda=0$ for any too-wide $\lambda$, as $s^0_\lambda=0$ for non-empty partitions $\lambda$.
\end{eg}
More examples for non-flag quiver flag varieties can be found below, illustrating the quantum rim-hook removal rule. 

We are now ready to prove Theorem \ref{thm:cohomology}.
\begin{proof} It is clear that the $s^i_\lambda$ generate the cohomology, when we do not restrict the size of the $\lambda$. Of course, if $\lambda$ is too long, $s^i_\lambda=0$. If $s^i_\lambda$ is too wide,  the proposition can be used to write it as a polynomial in elementary symmetric polynomials in other sets of variables, and some $s^i_{\lambda'}$ where each $\lambda'$ is a strictly thinner partition. By repeating this, one sees that the $s^i_\lambda$ fitting inside an $r_i \times (s_i-r_i)$ box generate the cohomology. 
\end{proof}
In particular, the theorem gives an efficient algorithm for computing products in the cohomology ring of the quiver flag variety in the basis $\mathcal{S}(Q)$: first compute products using the usual Littlewood--Richardson coefficients, then reduce any too-wide partitions via rim-hook removals, and repeat. 
\section{The Abelian/non-Abelian correspondence and quantum cohomology}
In this section, we briefly describe small quantum cohomology and the conjecture of \cite{CiocanFontanineKimSabbah2008}. 
\subsection{Gromov--Witten invariants and quantum cohomology}
We give a very brief overview of Gromov--Witten invariants and quantum cohomology (considering only small quantum cohomology in this paper). Let $Y$ be a smooth projective variety. Given $n \in \ZZ_{\geq 0}$ and $\beta \in H_2(Y)$, let  $M_{0,n}(Y,\beta)$ be the moduli space of genus zero stable maps to $Y$ of class $\beta$, and with $n$ marked points~\cite{Kontsevich95}. While this space may be highly singular and have components of different dimensions, it has a \emph{virtual fundamental class} $[M_{0,n}(Y,\beta)]^{virt}$ of the expected dimension~\cite{BehrendFantechi97,LiTian98}. There are natural evaluation maps $ev_i: M_{0,n}(Y,\beta) \to Y$ taking the class of a stable map $f: C \to Y$ to $f(x_i)$, where $x_i \in C$ is the $i^{th}$ marked point. There is also a line bundle $L_i \to M_{0,n}(Y,\beta)$ whose fiber at $f: C \to Y$ is the cotangent space to $C$ at $x_i$. The first Chern class of this line bundle is denoted $\psi_i$. Define:
\begin{equation}
  \label{eq:GW}
  \langle \tau_{a_1}(\alpha_1),\dots,\tau_{a_n}(\alpha_n) \rangle_{n,\beta} = \int_{[M_{0,n}(Y,\beta)]^{virt}} \prod_{i=1}^n ev_i^*(\alpha_i) \psi_i^{a_i}
\end{equation}
where the integral on the right-hand side denotes cap product with the virtual fundamental class.  If $a_i=0$ for all $i$, this is called a (genus zero) Gromov--Witten invariant and the $\tau$ notation is omitted; otherwise it is called a descendant invariant. It is deformation invariant. 

Now suppose that $Y$ is a smooth Fano variety.  The \emph{quantum cohomology ring} is defined by giving a deformation of the usual cup product of $H^*(Y)$ for every $t \in H^*(Y)$. The structural constants defining the new product are given by Gromov--Witten invariants.

Let $\{T_i\}$ be a homogeneous basis of $H^*(Y,\CC))$ and $\{T^i\}$ a dual basis. Let $t \in H^2(Y,\CC)$. The small quantum product is defined by 
\[\langle T^a \circ_t T^b, T^c\rangle:=\sum_{d \in H_2(Y)} e^{\int_{d}t} \int_{[M_{0,3}(Y,d)]^{virt}} ev_1^*(T^a) ev_2^*(T^b) ev_3^*(T^c).\]
The fact that $Y$ is Fano ensures that this sum is finite. 

  If $T_1,\dots,T_r$ are a basis of $H^2(Y,\CC)$, and $t_i$ a parameter for $T_i$, define $q_i:=e^{t_i}.$ For $d=\sum_{i=1}^r d_i T_i$, write $q^d=q_1^{d_1} \cdots q_r^{d_r}$. We can re-write the above as
\[\langle T^a \circ T^b, T^c\rangle:=\sum_{d \in H_2(Y)} q^d \int_{[M_{0,3}(X,d)]^{virt}} ev_1^*(T^a) ev_2^*(T^b) ev_3^*(T^c).\]
We view the quantum product as giving a product structure on the cohomology of $H^*(X)$ with coefficients in the polynomial ring in the $q_i$. 

\paragraph{\emph{The $J$-function}} We will need a generating function for descendent invariants called the \emph{$J$-function} in \S \ref{sec:proofconj}; we define it here (this is the small $J$-function). The $J$-function assigns an element of $H^*(Y) \otimes N(Y)[[z^{-1}]]$  (here $N(Y)$ is Novikov ring of $Y$) to every element of $H^0(Y)\oplus H^2(Y)$, as follows. The $J$-function is given by
\begin{equation}
  \label{eq:J}
J_Y(\tau,z):=e^{\tau/z}(1+z^{-1} \sum_i \langle\langle T_i/(z-\psi)\rangle\rangle T^i).
\end{equation}
Here $1$ is the unit class in $H^0(Y)$, $\tau \in H^2(Y)$, and
\begin{equation}
  \label{eq:J_correlator}
  \langle\langle T_i/(z-\psi)\rangle\rangle=\sum_{\beta \in \NE_1(Y)} q^\beta \sum_{a=0}^\infty\frac{1}{z^{a+1}} \langle \tau_a(T_i)\rangle_{1,\beta}.
\end{equation}
The small $J$-function satisfies
\[J_Y(\tau,z)=1+\tau z^{-1}+O(z^{-2}).\]
We sometimes consider the $J$-function to be $J_Y(0,z)$. 
Closed forms for the small $J$-function of toric complete intersections and toric varieties are known~\cite{Givental98}. If $Y$ has a torus action, one can define an equivariant $J$-function by replacing classes with the equivariant versions. 
\subsection{The conjecture of Ciocan-Fontanine--Kim--Sabbah}
Consider a GIT quotient as above, $X_G$, and its abelianisation $X_T$. Fix a lifting of $H^*(X_G)$ to $H^*(X_T)^W$. One can interpret the results of Martin as the statement that, for any $\alpha, \beta \in X_G$,
\[\tilde{(\alpha \cup \beta)} \cup \omega= \tilde{\alpha} \cup \tilde{\beta} \cup \omega.\]
For Grassmannians, as the authors of \cite{CiocanFontanineKimSabbah2008} observe, the naive generalisation of this statement to quantum cohomology is true, when one specialises the quantum product on the right hand side appropriately:
\[\tilde{(\alpha *_G \beta)} \cup \omega= \tilde{\alpha} *( \tilde{\beta} \cup \omega).\]
However, their conjecture is more complicated than this naive version, to a large part arising from the fact that in general, there is no natural relation between taking the cup product of $\tilde{\alpha}$ with $\omega$ and taking the quantum product  of $\tilde{\alpha}$ with $\omega$ (the conjecture is about big quantum cohomology, however, we restrict to small quantum cohomology). For the Grassmannian, these two products coincide. The key ingredient to the quantum rim-hook removal rule is that the choice of basis  $\mathcal{S}(Q)$ shares this special feature with the Grassmannian. Let us first state the general conjecture, and then explain how to specialise. Fix a basis of the cohomology of $X_G$, and let $t_i$ be the coordinates in this basis. Let $\tilde{t_i}$ be the coordinates on the fixed lift of $H^*(X_G)$ to $H^*(X_T)^W$. There are more curve classes on $X_T$ than $X_G$, however, there is a natural map $p$ specialising the quantum parameters of $X_T$ to those of $X_G$ (see \cite{CiocanFontanineKimSabbah2008} for the general case, or below for the quiver flag variety case). Let $*_G$ denote the quantum product in $X_G$, and $*$ the quantum product in $X_T$ specialised via $p$. For $\alpha, \alpha'$, there are unique $\zeta, \zeta'$ such that 
\[\tilde{\alpha}\cup \omega=\zeta * \omega, \hspace{1mm} \tilde{\alpha'} \cup \omega=\zeta' * \omega.\]
Then the conjecture of Ciocan-Fontanine--Kim--Sabbah is the following statement.
\begin{conjecture}[\cite{CiocanFontanineKimSabbah2008}]\label{conj:original} Let $\alpha, \alpha', \zeta, \zeta'$ be as above. Then 
\[(\tilde{(\alpha *_G \alpha')} \cup \omega)(t)=(\zeta*\zeta'*\omega)(\tilde{t}(t),0)\]
where $\tilde{t}(t)$ is an appropriate change of variables.  
\end{conjecture}
The authors of \cite{CiocanFontanineKimSabbah2008} observe that this change of variables is trivial when $X_G$ is a flag variety (this is Lemma 3.6.1(i) in the paper). This is because the flag variety has Fano index at least 2, which forces $\alpha*\omega=\alpha \cup \omega$ for any $\alpha\in H^2(X_T)$ . If $X_G$ is a Fano quiver flag variety, its Fano index can be one, but it is still true that $\alpha*\omega=\alpha \cup \omega$,  in fact for any $\alpha \in \mathcal{S}(Q)$. We will see this in the next section, by taking a closer look at the quantum cohomology ring of $X_T$, and then use it to simplify the conjecture considerably. In the final section, we will show that the conjecture holds for any Fano quiver flag variety. 
\begin{thm}\label{thm:conjecture} Conjecture \ref{conj:original} holds for any Fano quiver flag variety.
\end{thm}
The proof of this theorem involves a close look at the $I$-function of Fano quiver flag varieties; we delay the proof until the last section so that those uninterested in $I$-functions can more conveniently omit it. 
\subsection{The conjecture for Fano quiver flag varieties}
We first consider the quantum cohomology ring of a Fano toric quiver flag variety. In this case, the Batyrev ring relations can be expressed in terms of the quiver.
\begin{prop}[\cite{Batyrev93}] Let $M_\theta(Q,(1,\dots,1))$ be a toric Fano quiver flag variety, and set $R:=\QQ[q_1,\dots,q_\rho].$ Then
\[QH^*(M_\theta(Q,\br),\CC) \cong R[x_1,\dots,x_\rho]/I,\]
where $I$ is generated by the following elements, for each $i \in \{1,\dots,\rho\}$:
\[\prod_{\substack{a\in Q_1 \\ t(a)=i}} (-x_{s(a)}+x_{t(a)})-q_i \prod_{\substack{a\in Q_1 \\ s(a)=i}} (-x_{s(a)}+x_{t(a)}). \label{eqn:qabrelation}\]
Here we set $x_0=0$. If there are no outgoing arrows from vertex $i$, then the second summand above is understood to be $q_i$.
\end{prop}
From now on, let $X_G:=M_\theta(Q,\br)$ be a fixed choice of Fano quiver flag variety, and let $\mathcal{S}(Q)$ be the basis of $H^*(X_G)$ described in the previous section. Let $R:=\CC[q_1,\dots,q_\rho]$, and set $R_T:=\CC[q_{ij}: 1 \leq i \leq \rho, 1 \leq j \leq r_i]$.  We fix the choice of lifting of elements of $\mathcal{S}(Q)$ to $H^*(X_T)$ to be the natural one. The conjecture of \cite{CiocanFontanineKimSabbah2008} compares the quantum product of $X_G$ with a specialisation of that $X_T$, given by the map
$$p: R_T \to R, q_{ij} \mapsto (-1)^{r_i-1}q_i.$$
We define $*$ to be the product in $\overline{QH}(X_T):=R[x_{ij}:1 \leq i \leq \rho, 1 \leq j \leq r_i]/{I^{ab}_q}$ obtained from specialising the quantum product in $QH(X_T)$ via $p$. Applying the proposition, the ideal $I^{ab}_q$ is generated by, for each $ 1 \leq i \leq \rho, 1 \leq j \leq r_i$:
\[\prod_{\substack{a\in Q_1 \\ t(a)=i}}\prod_{k=1}^{r_{s(a)}} (-x_{s(a)k}+x_{t(a)j})-(-1)^{r_i-1}q_i \prod_{\substack{a\in Q_1 \\ s(a)=i}}\prod_{k=1}^{r_{t(a)}} (-x_{s(a)j}+x_{t(a)k}).\]
Expanding as in the classical case, we can re-write this as:
\begin{align} \label{eqn:qlowerdegree}
\sum_{k=0}^{s_i} \sum_{\sum k_a=k} (-1)^{k} x_{ij}^{s_i-k} \prod_{\substack{a \in Q_1 \\ t(a)=i}} s^{s(a)}_{(k_a)^t}-(-1)^{r_i-1}q_i \sum_{k=0}^{s_i'} \sum_{\sum k_a=k} (-1)^{s_i'-k} x_{ij}^{s_i'-k} \prod_{\substack{a \in Q_1 \\ t(a)=i}} s^{t(a)}_{(k_a)^t}.
\end{align}
Recall the definition of $s_i'$ from \eqref{eqn:si}. As before, the sum over $\sum k_a=k$ is understood to be the sum over all choices $0 \leq k_a \leq r_{s(a)}$ such that $\sum k_a=k.$  This equation means that quantum corrections enter when $x_{ij}$ appears with degree $s_i$ or higher. The following corollary is now immediate. 
\begin{cor}
Let $\alpha \in \mathcal{S}(Q)$. Then in $\overline{QH}(X_T)$,
\[\tilde{\alpha} * \omega = \tilde{\alpha} \cup \omega.\]
\end{cor}
\begin{proof}
The basis element $\alpha$ is a product of $s^i_{\lambda_i}$, for some $\lambda_i$ fitting into an $r_i \times (s_i-r_i)$ box. The $s^i_\lambda$ have no variables in common, so it suffices to consider each one separately. The restriction that $\lambda_i$ fits inside an $r_i \times (s_i-r_i)$  box ensures that in the product $\omega_i * s^i_{\lambda_i}$, $x_{ij}$ never appears with degree higher than $s_i-1$, and therefore, no quantum corrections are necessary. 
\end{proof}
This corollary implies both that the change of variables in the Conjecture \ref{conj:original} is trivial and that if $\alpha, \alpha' \in \mathcal{S}(Q)$, that $\zeta=\tilde{\alpha}$ and $\zeta'=\tilde{\alpha'}$; that is, the conjecture implies that 
\begin{equation}\label{eqn:lift}\tilde{(\alpha *_G \alpha')} \cup \omega=\tilde{\alpha}*\tilde{\alpha'}*\omega.\end{equation}
We can now obtain a quantum version of Martin's theorem (Theorem \ref{thm:martin} above). 
\begin{thm}\label{thm:arrowquantum}Let $X_G$ be a Fano quiver flag variety. Then
\[QH(X_G) \cong \overline{QH}(X_T)^W/I_q\]
where $I_q$ is the ideal generated by $f$ such that $\omega*f \in I^{ab}_q$, where $f$ is a polynomial in $\mathcal{S}(Q)$.
\end{thm}
\begin{proof}
Note that, as  can be seen by induction, for $\alpha_1,\dots,\alpha_k \in \mathcal{S}(Q)$,
\[\tilde{\alpha_1 *_G \cdots *_G \alpha_k} \cup \omega=\tilde{\alpha_1}* \cdots \tilde{\alpha_k} * \omega.\]
For a polynomial in elements of $\mathcal{S}(Q)$, with coefficients in $R$, we denote $P_G$ to be the element in $QH(X_G)$ obtained by interpreting the product as $*_G$. One can replace $\mathcal{S}(Q)$ with the lift, and multiply via $*$ -- we denote the result as $P_T \in \overline{QH}(X_T)$. Then for any such polynomial, 
\[\tilde{P_G}\cup \omega = P_T * \omega.\]
Now suppose $P_T * \omega \in I^{ab}_q$, and that $P_T$ is a polynomial in the lifted basis. Then for every element $\beta \in H^*(X_G)$, by Theorem \ref{thm:int}
\[\int_{X_G} P_G \cup \beta=\int_{X_T} \tilde{P_G} \cup \tilde{\beta} \cup \omega^2=\int_{X_T} (P_T*\omega) \cup \tilde{\beta} \cup \omega=0. \]
Therefore $P_G$ defines a relation in $QH(X_G)$. To see that all quantum relations are of this form, one can use the same argument as in the proof of Proposition 11.2.17 in \cite{coxkatz}: we can take the limit as $q_i\to 0$, where we know by Theorem \ref{thm:martin} these are all the relations. These relations must give all relations in a neighborhood of $q_i=0$, and therefore must be all the relations. 
\end{proof}

The quantum rim-hook removal now follows easily.
\begin{rem}[Notation] As above, recall that $s^0_\lambda=0$ for all $\lambda \neq \emptyset$. In addition, we also set $s^k_\lambda=0$ for all $\lambda \neq \emptyset$ when $k$ is not a vertex, i.e. when $k>\rho$, and $s^k_\emptyset=1$ for all $k \geq 0$. 
\end{rem}
\begin{thm}[Quantum rim-hook removal]\label{thm:qhookrim}
Let $\lambda$ a partition such that $\lambda_1 \geq s_i-r_i+1$. Then 
\begin{multline*} s^i_\lambda=\sum_{k=1}^{s_i} \sum_{\sum k_a=k} (-1)^{k+1} \HR^{s_i}(s_{\lambda+\{s_i-k\}}) \prod_{\substack{a\in Q_1 \\ t(a)=i}} s^{s(a)}_{(k_a)^t}\\
+(-1)^{r_i-1}q_i \sum_{k=0}^{s'_i} \sum_{\sum k_a=k} (-1)^{s_i'-k} \HR^{s_i}(s_{\lambda+\{s'_i-k\}}) \prod_{\substack{a\in Q_1 \\ s(a)=i}} s^{t(a)}_{(k_a)^t}.\end{multline*}
\end{thm}
\begin{proof} The proof is exactly as in the proof of Proposition \ref{prop:hookrim}.
\end{proof}
\begin{eg}
Suppose $M_\theta(Q,\br)$ is the flag variety of quotients $\CC^n$, $\Fl(n;r_1,\dots,r_\rho)$.  Then $s^i_\lambda$ is the Schur polynomial in the Chern roots of the $i^{th}$ tautological quotient bundle on the flag variety (the one of rank $r_i$). The quantum rim-hook removal rule in this case states that if $\lambda_1 \geq r_{i-1}-r_i+1$, then 
\begin{multline*} s^i_\lambda=\sum_{k=1}^{r_{i-1}} (-1)^{k+1} \HR^{r_{i-1}}(s_{\lambda+\{r_{i-1}-k\}})s^{i-1}_{(k)^t}\\+(-1)^{r_i-1} q_i \sum_{k=0}^{r_{i+1}}(-1)^{r_{i+1}-k}\HR^{r_{i-1}}(s_{\lambda+\{r_{i+1}-k\}})s^{i+1}_{(k)^t}.\end{multline*}
As a specific example, consider the flag variety $\Fl(4;2,1)$. Then 
\begin{align*}
s^1_{\yng(3)}&=q_1,\\
s^1_{\yng(3,1)}&=q_1 s^2_{\yng(1)},\\
s^1_{\yng(3,2)}&=q_1(s^1_{\yng(1)}s^2_{\yng(1)}-s^1_{\yng(1,1)}),\\
s^2_{\yng(2)}&=q_2+s^1_{\yng(1)}s^2_{\yng(1)}-s^1_{\yng(1,1)}.
\end{align*}
\end{eg}
\begin{eg} Consider the quiver flag variety
\begin{center}
  \includegraphics[scale=0.5]{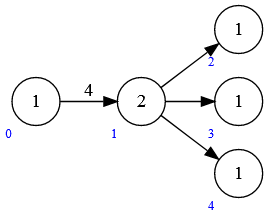}.
\end{center}
The vertex labels are in small print, in blue. 
Then 
\begin{flalign*}
s^1_{\yng(3)}&=q_1(s^2_{\yng(1)}s^3_{\yng(1)}+s^2_{\yng(1)}s^4_{\yng(1)}+s^3_{\yng(1)}s^4_{\yng(1)}-s^1_{\yng(1)}(s^2_{\yng(1)}+s^3_{\yng(1)}+s^4_{\yng(1)})+s^1_{\yng(2)}),\\
s^1_{\yng(3,1)}&=q_1(s^2_{\yng(1)}s^3_{\yng(1)}s^4_{\yng(1)}-s^1_{\yng(1,1)}(s^2_{\yng(1)}+s^3_{\yng(1)}+s^4_{\yng(1)})+s^1_{\yng(2,1)}),\\
s^1_{\yng(3,2)}&=q_1(s^1_{\yng(1)} s^2_{\yng(1)}s^3_{\yng(1)}s^4_{\yng(1)}-s^1_{\yng(1,1)}(s^2_{\yng(1)}s^3_{\yng(1)}+s^2_{\yng(1)}s^4_{\yng(1)}+s^3_{\yng(1)}s^4_{\yng(1)})+s^1_{\yng(2,2)}),\\
s^2_{\yng(2)}&=q_2+s^1_{\yng(1)}s^2_{\yng(1)}-s^1_{\yng(1,1)},\\
s^3_{\yng(2)}&=q_3+s^1_{\yng(1)}s^3_{\yng(1)}-s^1_{\yng(1,1)},\\
s^4_{\yng(2)}&=q_4+s^1_{\yng(1)}s^4_{\yng(1)}-s^1_{\yng(1,1)}.
\end{flalign*}
\end{eg}
\begin{eg} Consider the quiver flag variety
\begin{center}
  \includegraphics[scale=0.5]{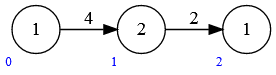}.
\end{center}
The rim-hook rule allows us to compute:
\begin{align*}
s^1_{\yng(3)}&=q_1(2 s^2_{\yng(1)}-s^1_{\yng(1)}),\\
s^1_{\yng(3,1)}&=q_1(s^2_{\yng(2)}-s^1_{\yng(1,1)}),\\
s^1_{\yng(3,2)}&=q_1(s^1_{\yng(1)} s^2_{\yng(2)}-2 s^1_{\yng(1,1)}s^2_{\yng(1)}),\\
s^2_{\yng(4)}&=q_2+2 s^1_{\yng(1)}s^2_{\yng(3)}-3 s^1_{\yng(1,1)} s^2_{\yng(2)}- s^1_{\yng(2)} s^2_{\yng(2)}+s^1_{\yng(2,1)}s^2_{\yng(1)}-s^1_{\yng(2,2)}.\\
\end{align*}
\end{eg}
\subsection{Relation to the Gu--Sharpe mirror}\label{sec:gusharpe}
In \cite{GuSharpe}, Gu--Sharpe construct a mirror for any Fano GIT quotient of a vector space. One key feature of a mirror, which is usually a Laurent polynomial, to a Fano variety is that the critical locus computes relations in the quantum cohomology ring. The Gu--Sharpe mirror is obtained by taking the Hori--Vafa mirror of the abelianised toric variety, together with the line bundles given by the roots $R$, and specialising the quantum parameters. We describe the construction here, as it gives different perspective on the signs appearing in the specialisation of the quantum parameters, the role of $\omega$, and is closely related to the $I$-function defined below. We then show that, as a corollary to Theorem \ref{thm:arrowquantum}, the Gu--Sharpe mirror indeed computes the quantum cohomology of a Fano quiver flag variety, providing evidence for the correctness of the proposal.

Recall the Hori--Vafa mirror of a toric variety. Let $X$ be a Fano toric variety constructed as the GIT quotient $\CC^m//T$, where $T$ is a torus of rank $r$, and the action of $T$ on $\CC^m$ is defined by weights $D_1,\dots,D_m \in \chi(T) \cong L$. Assume that the $D_i$ span the lattice $L$, and after re-ordering let $D_1,\dots,D_r$ be a basis. The Hori--Vafa mirror is defined by assigning a variable $x_i$ to each weight, and considering 
\[W:=x_1+\cdots+x_m.\]
The potential $W$ can be viewed as a map $(\CC^*)^m \to \CC$. The weights $D_1,\dots,D_m$ define a projection $\pi:(\CC^*)^m \to (\CC^*)^r$. Let $q=(q_1,\dots,q_r)$ be coordinates on the base, and $W_q$ the restriction of $W$ to a fiber of $\pi$. Using the choice of basis $D_1,\dots,D_r$, we can solve for $x_1,\dots,x_r$  in terms of the $q_i$ and the remaining variables, then write $W_q$ as a Laurent polynomial in the $x_{r+1},\dots,x_m$ and the $q_i$. We call $W_q$ the mirror to $X$.
\begin{eg}Consider $\prod_{i=1}^r \PP^{n-1}$, the abelianisation of the Grassmannian $\Gr(n,r)$. Then
\[W_q:=\sum_{i=1}^k x_{i1} + \cdots + x_{i n-1}+\frac{q_i}{x_{i1} \cdots x_{i n-1}}.\]
\end{eg}
By the mirror theorem for Fano toric varieties, the fibre-wise critical locus of $W_q$ computes the quantum cohomology relations of the toric variety. 

If $X_G$ is a Fano GIT quotient, the proposal of \cite{GuSharpe} is to consider the mirror $W_q$ not simply of the abelianisation $X_T$, but of $X_T$ together with the weights arising from the roots $R$ (i.e. of the total space $\bigoplus_{\alpha \in R} L_\alpha$ on $X_T$), and the specialise the quantum parameters. The proposal is that this is a mirror of $X_G$; in particular, they conjecture that the critical locus computes the quantum cohomology of $X_G$.  Let us give us the proposal explicitly in the case of a Fano quiver flag variety.  Let $X_G:=M_\theta(Q,\br)$, and let $X_T$ be the abelianisation. The weights of $X_T$ are given by the arrows of the abelianisation. Label the variables $x_a, a \in Q_1^{ab}$. There is a root for each pair of vertices $ij$ and $ik$; label the variable corresponding to this root  $y^i_{jk}, j \neq k.$ There should be a quantum parameter for each $q_{ij}$, but instead, we set $q_i=q_{ij}$ for all $j=1,\dots,r_i$. Then
\[W:=\sum_{a \in Q_1^{ab}} x_a +\sum_{i=1}^{\rho} \sum_{j=1}^{r_i} y^i_{jk}, \hspace{1mm} \frac{\prod_{a \in Q_1^{ab}, t(a)=ij} x_a}{\prod_{a \in Q_1^{ab}, s(a)=ij} x_a}\frac{\prod_{j \neq k} y^i_{kj}}{\prod_{j \neq k} y^i_{jk}}=q_i.\]
To solve for $W_q$, one must choose a lattice basis. Do this by first choosing an arrow for each vertex in $Q$: choose some $a_i, t(a_i)=i$. The lattice basis we choose for $\chi(T)$ is that given by the collection of arrows in $Q^{ab}$, $\mathcal{B}=\{a_{ij}\mid i \in \{1,\dots,\rho\}, j \in \{1,\dots,r_i\}\}$, where $a_{ij} \in Q_1^{ab}$ is the arrow from vertex $s(a)1$ to the vertex $ij$ associated to $a_i$. It's easy to see that this is indeed a lattice basis. One can then use the equations above to solve for $x_{a_{ij}}$ for each $ij$ (one might need to do several substitutions, but the process will end - the reason we make this choice is so that the quantum parameters match with those in the description above of the quantum cohomology); one will have
\[x_{a_{ij}}=q_i A_{ij} \frac{\prod_{j \neq k} y^i_{jk}}{\prod_{j \neq k} y^i_{kj}}\]
for some monomial $A_{ij}$ in the non-basis elements and $q_i$. We can then write $W_q$, the Gu--Sharpe mirror, by substituting the above for $x_{a_{ij}}$.

%%In fact, for each arrow $a$ in the abelianisation, there is a unique choice of path from $s(a)$ to $t(a)$ made up of only arrows in the basis or their reverse (this is just the statement that each $D_a$ can be written in the basis with coefficients $\pm 1$): label the set of basis arrows appearing in this path in the correct direction as $P_{a}^+$, and those appearing in the reverse direction as $P_{a}^-$. 
%% Then the mirror proposed by \cite{GuSharpe} is
%% \[W_q:=\sum_{a \in Q_1^{ab}, a \not \in \mathcal{B}} x_{a} +\sum_{i=1}^\rho \sum_{j=1}^{r_i} \frac{q_i \prod_{a \in Q_1^{ab}, a_{ij} \in P_{a}^-} x_{a}}{\prod_{a \in Q_1^{ab}, a_{ij} \in P_{a}^+} x_{a}}\frac{\prod_{j \neq k} y^i_{jk}}{\prod_{j \neq k} y^i_{kj}}.\]
 To compute the critical locus, begin by setting $y^i_{jk} \frac{\partial}{\partial y^i_{jk}}W_q=0$. It is straightforward to see that on the critical locus,
 \[y^i_{jk}=-y^i_{kj}=x_{a_{ij}}-x_{a_{ik}}.\]
The critical locus of $W_q$ thus agrees with that of
\[W_q:=\sum_{a \in Q_1^{ab}} x_a +\sum_{i=1}^{\rho} \sum_{j=1}^{r_i} y^i_{jk}, \hspace{1mm}x_{a_{ij}}=(-1)^{r_i-1} q_i A_{ij} .\]
together with the constraint that $x_{a_{ij}}-x_{a_{ik}} \neq 0$ (as $y^i_{jk}$ appears in the denominator). The sign appearing in front of the $q_i$ agrees with specialisation defined above. By the mirror theorem for toric varieties, the critical locus of this Laurent polynomial is exactly the ideal $I^{ab}_q$ appearing in Theorem \ref{thm:arrowquantum}. To see this directly requires a change of basis, matching the $x_a$ with the $x_{ij}$ appearing as generators in the theorem. 

But we also need to take into account the condition that $x_{a_{ij}}-x_{a_{ik}} \neq 0$. After this change of basis, this becomes the condition that $x_{ij}-x_{ik} \neq 0$ for all $j \neq k$. As $\omega$ is the product of such factors, this can be translated as the condition that one can divide by $\omega$. Gu--Sharpe conjecture (for any Fano GIT quotient of a vector space, not just a quiver flag variety as described here) that the quantum cohomology ring relations are generated by Weyl invariant polynomials $f$ such that $\omega f$ lie in the Jacobi ideal of mirror. Comparing this with Theorem \ref{thm:arrowquantum}, we see that the proof of this conjecture follows as an immediate corollary. 
\section{Proof of the conjecture for Fano quiver flag varieties}\label{sec:proofconj}
In \cite{CiocanFontanineKimSabbah2008}, they prove Conjecture \ref{conj:original} for flag varieties. Moreover, they show in Theorem 4.3.6 that if $X$ is any Fano variety of index at least 2, with a sufficiently nice torus action, then the conjecture follows from the Abelian/non-Abelian correspondence for small (equivariant) $J$-functions. The condition on Fano index is used to apply Lemma 3.6.1(i) in the paper: above, we have shown that the statement of the lemma holds for any Fano quiver flag variety. The condition on Fano index can therefore be dropped.

The Abelian/non-Abelian correspondence for small equivariant $I$-functions is known for Fano GIT quotients of vector spaces, including Fano quiver flag varieties \cite{Webb2018}. The equivariant $I$-function is another generating series, related to the quasi-map theory of $M_Q$.  Therefore, Conjecture \ref{conj:original} will follow from showing that $I=J$ for Fano quiver flag varieties, and showing that there is a nice torus action. We'll first describe the torus action.
 
 \begin{mydef} Let $M_\theta(Q,\br)$ be a quiver flag variety. Define $T_Q:=\prod_{a \in Q_1} \CC^*$. The torus $T_Q$ acts on $\Rep(Q,\br)=\oplus_{a \in Q_1} \Mat(r_{s(a)},r_{t(a)})$ in the obvious way. This action commutes with the action of $G$, and so defines an action of $T_Q$ on $M_\theta(Q,\br)$.
 \end{mydef}
 \begin{rem} This is a generalisation of an action defined in \cite{Webb2018} for a particular quiver flag variety.
 \end{rem}
 \begin{rem}The action of $T_Q$ is not necessarily effective on $M_\theta(Q,\br)$. If we take the action of $T_Q/(G\cap T_Q)$ in the case when $Q$ is a flag variety, this gives the usual torus action. 
 \end{rem}
 
 The type of torus action required is one where the fixed points are isolated. We show that $T_Q$ has this property.
 \begin{prop}\label{prop:isolated} The fixed points of of the $T_Q$ action on $M_\theta(Q,\br)$ are isolated.
 \end{prop}
 \begin{proof}
 There is a natural embedding of $M_\theta(Q,\br)$ into a product of Grassmannians $Y=\prod_{i=1}^\rho Gr(\tilde{s_i},r_i)$ (see \cite{kalashnikov}). Here $\tilde{s_i}$ is the number of paths from $0 \to i$.  There is an induced action of $T_Q$ on $V_Y=\prod_{i=1}^\rho \Mat(r_i \times \tilde{s_i})$ which restricts on the quiver flag variety to the action defined above. Let $(A_i)_{i=1}^\rho \in V_Y$ be a point in $V_Y$.  We think of the columns of $A_i$ as being indexed by paths from $0 \to i$. Then the action of $(t_a)_{a \in T_Q}$ on $(A_i)_{i=1}^\rho$ is given by multiplying the $kp^{th}$ entry of $A_i$ by $\prod_{a \in p} t_a$. The factor multiplying each column of $A_i$ is distinct, as it corresponds to a distinct path. Fixed points of the $T_Q$ action on $Y$ are thus isolated, and thus so are the fixed points for the restricted action. 
 \end{proof}
 
 Fix a Fano quiver flag variety $M_Q:=M_\theta(Q,\br)$, and let $M_{Q^{ab}}$ denote its abelianisation. In \cite{Webb2018}, the author gives a closed formula for the equivariant $I$-function of a GIT quotient of the form $V//G$; we now define this formula. Our interest in the $I$-function arises in the fact that it coincides with the small equivariant $J$ function for any Fano GIT quotients up to a mirror map; we will show that this mirror map is trivial for $M_Q$, so that in this case $I=J$. Let $\lambda_a$ be the equivariant parameters for the $T_Q$ action. The coefficients of the $I$-function are built out of two factors. Recall that $D_a$ is the divisor associated to the arrow of $a\in Q_1$. 

Given $\tilde{d} \in \NE_1(M_{Q^{\ab}})\cong \prod_{i=1}^\rho \prod_{j=1}^{r_i} \ZZ_{\geq 0}$,
\[
I_{Q^{ab}}(\tilde{d})=\frac{\prod_{a \in Q_1^{ab}} \prod_{m \leq 0} (D_a+\lambda_a +m z)}{\prod_{a\in Q^{ab}_1} \prod_{m \leq \langle \tilde{d}, D_a \rangle} (D_a+\lambda_a+m z)},
\]
and
\[
I_{R}(\tilde{d})=\frac{\prod_{\alpha \in R} \prod_{m \leq \langle \tilde{d}, D_\alpha \rangle} (D_\alpha+m z)}{\prod_{\alpha \in R} \prod_{m \leq 0} (D_\alpha+m z)}.
\]
Recall that $R$ is the set of roots. Notice that all but finitely many factors cancel in products. We can write
\[I_{R}(\tilde{d})=\prod_{i=1}^{\rho}(-1)^{\epsilon(d)}\frac{\prod_{\alpha \in R^+} (D_\alpha+\langle \tilde{d}, D_\alpha \rangle z)}{\prod_{\alpha \in R^+}D_\alpha}.\]
  Here $d$ is the image of $\tilde{d}$ under the projection $\NE_1(M_{Q^{ab}}) \to \NE_1(M_Q)$ and $\epsilon(d)=\sum_{i=1}^\rho d_i(r_i-1)$. The projection is given explicitly as $(d_{ij})_{ij} \mapsto (\sum_{j=1}^{r_i} d_{ij})_{i=1}^\rho$. The product
  \[I_{Q^{ab}}(\tilde{d}) \prod_{\alpha \in R}\prod_{\alpha \in R^+} (D_\alpha+\langle \tilde{d}, D_\alpha \rangle z)\]
is $W$-anti-invariant and hence divisible by $\omega$. 

Let $\tau \in H^0(M_Q)\oplus H^2(M_Q)$. Define the $I$-function of $M_Q$ to be
\[
I_{M_Q}(\tau,z) =e^{\tau/z} \sum_{d \in \NE_1(M_Q)} \sum_{\tilde{d} \to d} I_{Q^{ab}}(\tilde{d})I_{R}(\tilde{d}).
\]
\begin{rem}This is the $I$-function one would obtain for the toric variety $M_{Q^{ab}}$ together with the line bundles $L_\alpha, \alpha \in R$: precisely the same process as gives the Gu--Sharpe mirror. 
\end{rem}
We set $I_{M_Q}(z)=I_{M_Q}(0,z)$. Note that $I_{T_{M_Q}}(\tilde{d})$ is homogeneous of degree 
\[\langle -K_{M_{Q^{ab}}},\tilde{d}\rangle=\langle-K_{M_Q},d\rangle,\] so defining the grading of $q^d$ to be $\langle-K_{M_Q},d\rangle$, $I_{M_Q}(z)$ is homogeneous of degree $0$. As $M_Q$ is Fano, we can write $I_{T_{M_Q}}(\tilde{d})$ as
\begin{equation}\label{eqn:powerseries} z^{\langle-K_{M_Q},\tilde{d}\rangle}(b_0+b_1 z^{-1}+b_2 z^{-2} +\cdots), b_i \in H^{2i}(X).\end{equation}
If $\langle -K_{M_{Q^{ab}}},\tilde{d}\rangle \geq 2$ for all effective $\tilde{d}$ (a property sometimes called being of Fano index at least 2), we see immediately that the the coefficient of $z^{-1}$ vanishes in $I_{M_Q}$. We now show that any Fano quiver flag variety has this property.
\begin{prop} The coefficient of $z^{-1}$ vanishes in the power series expansion of $I_{M_Q}(z)$.
\end{prop}
\begin{proof}
By the above discussion, a $z^{-1}$ term appears only if $\langle -K_{M_{Q^{ab}}},\tilde{d}\rangle = 1$ for some effective $\tilde{d}$. If $\tilde{d}=(d_{ij})$, then it isn't hard to see that  
\[\langle -K_{M_Q},\tilde{d}\rangle =\sum_{i=1}^\rho \sum_{j=1}^{r_i} d_{ij}(s_i-s_i').\]
Since $M_Q$ is Fano, and $Q$ is graph reduced, $s_i-s_i'>0$ for all $i$, so the only way this pairing is $1$ is if $\tilde{d}$ has exactly one non-zero entry, which is one. Say this entry is at place $ij$, so $\tilde{d}=e_{ij}$. Then we can consider  
\[I_{Q^{ab}}(\tilde{d})I_{R}(\tilde{d})=\prod_{a \in Q^{ab}_1, t(a)=ij} \frac{1}{D_a+\lambda_a+z} \prod_{a \in Q^{ab}_1, s(a)=ij}D_a \prod_{k \neq j} \frac{D_{ij}-D_{ik}+z}{-D_{ij}+D_{ik}}. \]
Expanding $\frac{1}{D_a+\lambda_a+z}$ as a power series, we see that the lowest degree term in $z^{-1}$ has degree at least
\[s_i -(r_i-1) \geq 2,\]
where the inequality follows from the fact that $s_i-r_i >0$. 
\end{proof}
Therefore the asymptotic expansion of $I_{M_Q}(z)$ is
\[1+O(1/z^2),\]
and by \cite{quasimap} we have $I(0,z)=J(0,z)$. We can now prove Theorem \ref{thm:conjecture}.
\begin{proof}
By \cite{Webb2018} the Abelian/non-Abelian correspondence holds for small $I$-functions, as $I=J$ it therefore also holds for small $J$-functions. By \cite{CiocanFontanineKimSabbah2008}, this implies Conjecture \ref{conj:original}.
\end{proof}

\bibliographystyle{amsplain}
\bibliography{bibliography}
\end{document}